\documentclass[11pt]{article}%

\usepackage{a4wide}
\usepackage{amsmath,amssymb,amsthm}
\usepackage{enumerate}
\usepackage{tikz-cd}
\usepackage{multicol}
\usepackage{hyperref}
\hypersetup{
    colorlinks=true,  
    linkcolor=blue, 
    citecolor=red,   
    urlcolor=magenta 
}
\usepackage[nameinlink]{cleveref}
\usepackage{mathtools}
\usepackage[normalem]{ulem} 
\usepackage{graphicx}

\newtheorem{lemma}{Lemma}[section]

\newtheorem{theorem}[lemma]{Theorem}
\newtheorem{observation}[lemma]{Observation}
\newtheorem{conjecture}[lemma]{Conjecture}

\newcommand{\R}{{\mathbb{R}}}

\renewcommand{\AA}{{\mathcal{A}}}
\newcommand{\CC}{{\mathcal{C}}}
\newcommand{\FF}{{\mathcal{F}}}
\newcommand{\KK}{{\mathcal{K}}}
\newcommand{\LL}{{\mathcal{L}}}

\DeclareMathOperator{\dist}{dist}

\newcounter{sideremark}

\title{Non-acyclic spaces of line transversals\thanks{Supported by ERC-CZ project LL2328 of the Ministry of Education of Czech Republic.}}
\author{Haochi Jiang \and Martin Tancer}

\date{\small Department of Applied Mathematics, Faculty of Mathematics and
Physics, \\ Charles University, Czech Republic}
\begin{document}

\maketitle
\begin{abstract}
Cheong, Goaoc and Holmsen conjectured that every connected component of the space of line
transversals to a family of pairwise disjoint open convex sets in $\R^d$ is acyclic. We disprove
this conjecture by showing that the homology may be nontrivial in any fixed
  dimension provided that $d$ is large enough. More precisely, we show that for
  every $n \geq 1$ there is a finite family of pairwise disjoint open convex sets in
  $\R^{3n}$ such that the $(n-1)$st homology (over an arbitrary ring) of the
  space of line transversals to this family is nonzero. 
\end{abstract}

\section{Introduction}

  Let $\FF$ be a family of subsets of $\R^d$. A line $\ell$ in $\R^d$ is 
  \emph{transversal} to $\FF$ if $\ell$ meets every member of $\FF$. A
  collection of all lines transversal to $\FF$ with a natural topology on it 
  forms a topological space called the \emph{space of transversals} to $\FF$.
  (The topology on this space is explained in the beginning of
  Section~\ref{s:compact}.)

 Cheong, Goaoc and Holmsen~\cite[Theorem~1.3]{cgh24} proved that for any family
 $\FF$ of at least two disjoint open convex sets in $\R^3$, every connected
 component of the space of line transversals to $\FF$ is acyclic. Furthermore,
 they conjectured that their result generalizes to higher dimensions:

\begin{conjecture}[Conjecture~4.5 of~\cite{cgh24}]
  Let $\FF$ be any family of (at least two) pairwise disjoint open convex sets in
  $\R^d$. Then every connected component of the space of line transversals to
  $\FF$ is acyclic.
\end{conjecture}

We disprove this conjecture. As an intermediate step, we first disprove a
variant of the conjecture with a (finite) family of compact convex sets. (In
the theorems below, by $\tilde H_{n-1}(\cdot)$ we mean the reduced $(n-1)$st homology over
an arbitrary ring.)

\begin{theorem}
\label{t:compact_intro}
  For every $n \geq 1$ there is a finite family $\CC$ of pairwise disjoint
  compact convex set in $\R^{3n}$ such that $\tilde H_{n-1}(\LL)\ne 0$ where $\LL$ is
  the space of line transversals to $\FF$. 
\end{theorem}

Then, we inflate the sets in the construction required to prove
Theorem~\ref{t:compact_intro} and we obtain counterexamples to the original
conjecture.

\begin{theorem}
\label{t:inflated_intro}
  For every $n \geq 1$ there is a finite family $\CC$ of pairwise disjoint
  open (bounded) convex set in $\R^{3n}$ such that $\tilde H_{n-1}(\LL)\ne 0$ where $\LL$ is
  the space of line transversals to $\FF$.
\end{theorem}

\paragraph{Remarks.} The theorems above provide non-acyclic families only if
$n \geq 2$. The case $n=1$ is probably well known. However, we keep $n = 1$ in
the statement as it fits into a common framework and, in addition, the case $n =
1$ provides a partial motivation for the construction for $n \geq 2$.

One may be interested in a question what is the homotopy type of the spaces
$\LL$ we construct in order to prove Theorems~\ref{t:compact_intro}
and~\ref{t:inflated_intro}. We believe that in both cases, $\LL$ is homotopy
equivalent to $S^{n-1}$. However, in both cases, it is easier just to show
$\tilde H_{n-1}(\LL)\ne 0$. Without a direct application, we prefer a simpler
proof.

One may also be interested whether analogues of Theorems~\ref{t:compact_intro}
and~\ref{t:inflated_intro} hold for the space of $k$-transversals (that
is transversal $k$-flats). In fact, as discussed below Conjecture~4.5
of~\cite{cgh24}, for $k \geq 3$ the space of $k$-transversals in $\R^d$ to a
collection of disjoint open convex sets can have arbitrarily complicated
topology (provided that $k \leq d-3$). For $k = 2$, in the case of compact
sets, we can easily generalize the construction from Theorem~\ref{t:compact_intro}.
Indeed, we consider $\R^{3n}$ as a subspace of $\R^{3n+1}$, and we take $\CC$
from Theorem~\ref{t:compact_intro}. We obtain $\CC'$ by adding a point $p \in
\R^{3n+1}\setminus \R^{3n}$ to $\CC$. Every transversal plane to $\CC'$ meets
$\R^{3n}$ in a line transversal to $\CC$ and vice versa, every transversal line
to $\CC$ can be uniquely extended to a plane transversal to $\CC'$. It follows
that the space of transversal lines to $\CC$ is homeomorphic to the space
of transversal planes to $\CC'$. We believe that this construction can be
`inflated' to the setting of collections of disjoint open convex set in a
similar way how we deduce Theorem~\ref{t:inflated_intro} from
Theorem~\ref{t:compact_intro}. However, we leave out the details.

\paragraph{Motivation and background.}
Line transversals (or in general $k$-transversals) have been intensively 
studied mainly in the context of Helly type results. For example, the famous
Hadwiger's transversal theorem~\cite{hadwiger57} states that a family of
convex sets in a plane admits a line transversal if and only if every three
members of the family admit a line transversal with consistent order. This has
been further generalized by Goodman, Pollack and Wenger~\cite{gp88,pw90,wenger90} to transversals of
codimension 1. This is still an area of active research as demonstrated by a
recent breakthrough by McGinnis and Sadovek~\cite{ms26} generalizing the result
to arbitrary $k$-transversals.

For the space of line transversals (or $k$-transversals in general), as
explained in~\cite{cgh24}, the motivation to study its topology stems from
Helly-type results for sets of bounded topological complexity~\cite{cgg14,
montejano14, gpptw17,abg26}. In this context, our
Theorems~\ref{t:compact_intro} and Theorems~\ref{t:inflated_intro} should be
understood as negative results in the sense that the topology of the 
space of line transversals may be more complicated than expected. 

Borcea, Goaoc and Petitjean~\cite{bgp08} considered directions of (directed) line
transversals to a collection of disjoint balls in $\R^d$. In this case, they proved that
they from a convex subset of $S^{d-1}$, in particular, it is contractible.

\section{Proof of Theorem~\ref{t:compact_intro}}
\label{s:compact}

The aim of this section is to prove Theorem~\ref{t:compact_intro}.

\paragraph{The space of lines.}
First we clarify the topology on the space of transversals.
Let $d$ be a positive integer. By the \emph{space of lines} $\LL(\R^d)$ in
$\R^d$ we mean the quotient space obtained from the deleted product
$(\R^d)_\Delta := \{(x,y) \in \R^d \times \R^d \colon x \neq y\}$ by considering
$(x,y)$ and $(x', y')$ equivalent, if they span the same line. Similarly, given
a family $\FF$ of subsets of $\R^d$, we regard the space of transversals to
$\FF$
as a subspace of $\LL(\R^d)$ such that the class of a point $(x,y) \in
(\R^d)_\Delta$ belongs to this subspace if and only if the line spanned by $x$
and $y$ is transversal to $\FF$. We will need the following observation.

\begin{observation}
\label{o:project_to_F}
  Let $F$ be an affine flat in $\R^d$ and let $\AA \subseteq \LL(\R^d)$ be a
  subspace such that each line in $\AA$ intersects $F$ in exactly one point. Then
  the map $f \colon \AA \to F$ given by $\ell \to \ell \cap F$ is continuous.
\end{observation}

\begin{proof}
  Let $P \subseteq (\R^d)_\Delta$ be the preimage of $\AA$ under the quotient
  map $(\R^d)_\Delta \to \LL(\R^d)$. The map $g \colon P \to F$ that maps
  $(x,y) \in P$ to the intersection of $F$ and the line spanned by $x$ and $y$
  is continuous. In addition, if $(x,y)$ and $(x',y')$ in $P$ define the same
  line, then $g(x,y) = g(x',y')$. Thus, by the universal property of quotient
  spaces, $g$ is continuous.
\end{proof}

\paragraph{Proof of Theorem~\ref{t:compact_intro} for $n=1$.}
Now we prove Theorem~\ref{t:compact_intro} for $n=1$. This proof serves as a
motivation for what we want to do for general $n$. In addition, we introduce a
lemma that will be reused for general $n$.

From now on, let $P = \{(x,y,xy) \colon x,y \in \R\} \subseteq \R^3$. This is a
doubly ruled hyperbolic paraboloid. It contains lines $\ell_r := \{(t,r,tr)\colon t
\in \R\}$ for $r \in \R$ and $m_r := \{(r, t, tr)\colon t  \in \R\}$ for $r \in
\R$ and no other lines. It is easy to check that $\ell_r \cap m_{r'} = (r',r,rr')$ for $r,r' \in
R$ and that $\ell_r \cap \ell_{r'} = m_r \cap m_{r'} = \emptyset$ for
distinct $r, r' \in \R$.

\begin{lemma}
\label{l:3lines}
  Let $\ell$ be a line in $\R^3$ and $a,b,c$ be
  three distinct real numbers. Then $\ell$ is transversal to $\ell_a$,
  $\ell_b$ and $\ell_c$ if and only if $\ell  = m_r$ for some $r \in \R$; see
  Figure~\ref{f:paraboloid}.
\end{lemma}

\begin{figure}
  \centering
    \includegraphics{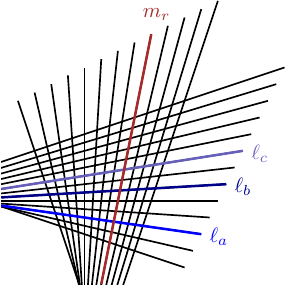}
    \caption{The line transversals to $\ell_a$, $\ell_b$ and $\ell_c$.}
    \label{f:paraboloid}      
\end{figure}

\begin{proof}
  If $\ell  = m_r$, then it meets $\ell_a, \ell_b$ and $\ell_c$ by
  considerations above the statement of the lemma.

   On the other hand, assume that $\ell$ is transversal to $\ell_a$, $\ell_b$
   and $\ell_c$. The lines $\ell_a, \ell_b$ and $\ell_c$ lie in $P$, and they are pairwise
   disjoint, thus $\ell$ meets $P$ in at least three points. Because $P$ is a
   quadric, this means that $\ell$ lies in $P$. The line $\ell$ cannot be any
   of $\ell_r$, thus it must be $m_r$ for some $r \in R$.
\end{proof}

With the lemma above, it is easy to construct an example of a family $\CC$
proving Theorem~\ref{t:compact_intro} for $n=1$. Indeed, we start with three
disjoint lines on $P$, for example $\ell_0$, $\ell_1$ and $\ell_2$. We pick two
points $p,q$ on $P$ avoiding $\ell_0$, $\ell_1$ and $\ell_2$ so that the line
$\ell$ determined by $p$ and $q$ does not belong to $P$, that is, it intersects
$P$ only in $p$ and $q$. For example, we can pick $p = (0,3,0)$ and $q =
(1,4,4)$. Then, by Lemma~\ref{l:3lines}, the transversal lines to $\ell_0$,
$\ell_1$, $\ell_2$ and $\ell$ are those lines $m_r$ which pass through $\ell$.
As $m_r$ is full contained in $P$, there are exactly two such lines, one
passing through $p$ and another one through $q$. For our specific choices,
these are lines $m_0$ and $m_1$. Thus the space of line transversals to
$\ell_0$, $\ell_1$, $\ell_2$ and $\ell$ is a 2-point discrete set. In this
example $\ell_0$, $\ell_1$, $\ell_2$ and $\ell$ are not compact convex sets but
we can restrict those lines to segments intersecting $m_0$ and $m_1$, and we are
done.

\paragraph{Construction for Theorem~\ref{t:compact_intro} for general $n$.}
Now we want to generalize the construction above to any $n \geq 2$. Let us
first sketch an idea. Our original approach was to generalize $P$ to some
algebraic variety $V$ of codimension $1$ such that it would contain some affine
flats $F_a$ playing the role of $\ell_a$ in the previous example so that the line transversals would belong to $V$. Then we could take one more flat $F$ (playing
the role of $\ell$ in the previous proof) and intersect it with $V$ in
something topologically nontrivial (hopefully $S^{n-1}$). Then the space of the
line transversals to three (or a finite number) of such flats $F_a$ and to $F$
would be the topology of $S^{n-1}$ essentially by
Observation~\ref{o:project_to_F}. A difficulty with this approach is that the
examples of $V$ we tried turned out to be too flat in the sense that we were
not able to intersect $V$ with a flat (or convex set in general) in anything
topologically more complicated than $S^0$.

We managed to circumvent this difficulty by considering $V$ of higher
codimension. In the actual proof we do not need to work with such a $V$
explicitly. But $V$ we implicitly use is $P^n \subseteq (\R^3)^n$.

  Let $n$ be a positive integer. Given $a \in \R$ and $i \in [n]$, by $F^i_a$
  we denote the flat $\R^{3(i-1)}\times \ell_a \times\R^{3(n-i)} \subseteq
  \R^{3n}$.\footnote{Formally speaking, $F^i_a$ also depends on $n$. However, we
  will always assume that $n$ is known from the context.} We also define 
  \begin{equation}
    \label{e:F0}
    F_0:=\left\{(x_1,x_1,z_1,x_2, x_2, z_2, \dots, x_n,x_n,z_n) \in \R^{3n}
    \colon \sum_{i=1}^n (z_i - 2x_i + 1)= \frac12\right\}.
  \end{equation}

\begin{lemma}\label{l:general}
  Let $n$ be a positive integer and let $(a_1^1,a_2^1,a_3^1,\dots,a_1^n,a_2^n,a_3^n)$ be a $3n$-tuple of real
   numbers such that $a_1^i$, $a_2^i$ and $a_3^i$ are pairwise distinct for every $i \in
   [n]$. 

  Then all transversal lines to $\FF :=\{F_0\}\cup \{F^i_{a_j^i}\colon
  i\in [n],j\in [3] \}$ are of the form
\[
\{(r_1,r_1,r_1^2,\dots,r_i,r_i,r_i^2,\dots,r_n,r_n,r_n^2)+(0,q_1 t,q_1 r_1
  t,\dots,0,q_i t,q_i r_i t,\dots ,0,q_n t,r_n q_n t)\}
\]
where $r_1,\dots,r_n\in \R$ are such that $\sum_{i=1}^n(r_i-1)^2=\frac{1}{2}$ and
  $q_i\in \R\setminus{\{0\}}$. 
\end{lemma}

\begin{proof}
  Assume that $\ell$ is a line transversal to $\FF$. Given $i \in [n]$,
  let $\pi_i \colon \R^{3n} \to \R^3$ be the projection to the $i$th triplet of coordinates in
  $\R^{3n}$.
  We get that $\pi_i(\ell)$ is transversal to
  $\pi_i(F^i_{a_j^i})=\ell_{a_j^i}$ for $i \in [n]$ and $j \in [3]$. In
  particular, $\pi_i(\ell)$ must be a line.
  By Lemma~\ref{l:3lines}, $\pi_i(\ell)=m_{r_i}$ for some $r_i\in \R$.
  Therefore, $\ell$ belongs to the $n$-flat
\[
 \rho:= \prod_{i=1}^n
  m_{r_i}=\{(r_1,t_1,r_1t_1,\dots,r_i,t_i,r_it_i,\dots,r_n,t_n,r_nt_n)\colon
  t_1,\dots, t_n \in \R\}\subseteq
  \R^{3n}.
\]
  
The only intersection of $\rho$ and $F_0$ is the point
  $(r_1,r_1,r_1^2,\dots,r_i,r_i,r_i^2,\dots,r_n,r_n,r_n^2)$, and this occurs only
  if $\sum_{i=1}^n (r_i-1)^2=\frac{1}{2}$ due to the condition in the
  definition of $F_0$. Hence, $\ell$ has to pass through the point
  $(r_1,r_1,r_1^2,\dots,r_i,r_i,r_i^2,\dots,r_n,r_n,r_n^2)$, and since
  $\ell\subseteq\rho$, 
  \[ \ell=
  \{(r_1,r_1,r_1^2,\dots,r_i,r_i,r_i^2,\dots,r_n,r_n,r_n^2)+(0,q_1 t,q_1 r_1
  t,\dots,0,q_i t,q_i r_i t,\dots,0,q_n t,r_n q_n t)\} \] 
  for some $q_1, \dots,
  q_n \in \R$. Moreover, if $q_i=0$ for some $i \in [n]$, then $\ell$ cannot be 
  transversal to $F^i_{a_1^i}$, $F^i_{a_2^i}$ and
  $F^i_{a_3^i}$ simultaneously. Thus all $q_i$ are non-zero. 
\end{proof}

\begin{lemma}\label{l:disjoint_sets_general}
  For each $i\in[n]$, let $G_i$ be the halfspace
  $\{(x_1,y_1,z_1,\dots,x_n,y_n,z_n)\in  \R^{3n}\colon y_i\ge 6\}$ and let $Q
  := [0,8^{n+1}]^{3n}$.
  For each $i\in [n]$ and $a\in \{3,4,5\}$, let  $C^a_i := F_i^a \cap
  \bigcap_{j=1}^{i-1} G_j \cap Q $ and

\[
  C_0:=\left\{(x_1,x_1,z_1,x_2,x_2,z_2,\dots,x_n,x_n,z_n) \in Q\colon x_i\leq 1.9\hbox{
    for $i \in [n]$ and }\sum_{i=1}^n( z_i - 2 x_i + 1) = \frac12\right\}
\]

  Then $\CC=\{C_0\}\cup \{C_i^a\colon i\in [n],a\in \{3,4,5\}\}$ is a family of
  pairwise disjoint compact convex sets, and the set of lines transversal to $\CC$ contains the lines
  \[
    \{(r_1,r_1,r_1^2,r_2,r_2,r_2^2,\dots,r_n,r_n,r_n^2)+(0,8^{n-1}t,8^{n-1}r_1
    t, 0,8^{n-2}t,8^{n-2}r_2 t,\dots,0,t,r_n t)\colon t \in \R\}
    \]
where $r_1,\dots,r_n\in \R$ are such that $\sum_{i=1}^n(r_i-1)^2=\frac{1}{2}$.
\end{lemma}

\begin{proof}
  Each set in $\CC$ is an intersection of half-spaces, therefore convex. In
  addition, each such set is a closed set contained in $Q$, therefore compact.
  By
  checking the second coordinate, we get that $C_0$ is disjoint from $C_i^a$
  for every $i \in [n]$ and $a \in \{3,4,5\}$. Indeed, the second coordinate of
  every point in $C_0$ is at most $1.9$ while it is $a$ if $i = 1$ and at least
  $6$ if $i \geq 2$ as $C_i^a \subseteq G_1$ for $i \geq 2$. Similarly, we
  check that $C_i^a$ and $C_j^b$ are disjoint if $i < j$, or $i =j$ and $a \neq
  b$. Indeed, the coordinate corresponding to $y_i$ (in the expression for
  $G_i$) of every point in $C_i^a$ is equal $a$ because $C_i^a \subseteq
  F_i^a$. On the other hand, for every point in $C_j^b$, this coordinate equals
  $b$ if $i = j$, and it is at least $6$ if $j > i$ as $C_j^b \subseteq G_i$ in the
  latter case.
  
  Now we consider a line 
  \[
    \ell := \{(r_1,r_1,r_1^2,r_2,r_2,r_2^2,\dots,r_n,r_n,r_n^2)+(0,8^{n-1}t,8^{n-1}r_1
    t, 0,8^{n-2}t,8^{n-2}r_2 t,\dots,0,t,r_n t)\colon t \in \R\}
    \]
where $r_1,\dots,r_n\in \R$ are such that $\sum_{i=1}^n(r_i-1)^2=\frac{1}{2}$.
  This, in particular, implies that $r_1, \dots, r_n \in [1 - \sqrt 2/2, 1 +
  \sqrt2 /2]$.
We need to check that $\ell$ is indeed transversal to $\CC$.

The point $(r_1,r_1,r_1^2,\dots,r_n,r_n,r_n^2)$ belongs to $\ell$. It is also
  easy to check that it belongs to $C_0$. Indeed, we need that $0 \leq 
  r_i \leq 1.9$ and $0 \leq r_i^2 \leq 8^{n+1}$ for
  every $i \in [n]$, and that $\sum_{i=1}^n (r_i^2 -2r_i + 1) = \frac12$.
  All these conditions follow from the aforementioned properties of $r_i$.

  Now for $i\in [n]$ and $a \in \{3,4,5\}$ we set $t=\frac{a-r_i}{8^{n-i}}$
  into the definition of $\ell$. We get that the point
 \[
  (r_1,r_1+8^{i-1}(a-r_i),r_1^2+8^{i-1}r_1(a-r_i),\dots,r_i,a,ar_i,\dots,r_n,r_n+8^{i-n}(a-r_i),r_n^2+8^{i-n}r_n(a-r_i))
  \]
belongs to $\ell$. By the definition of $F_i^a$, this point belongs to $F_i^a$. 
In addition, for $j \in [i-1]$, $r_j+8^{i-j}(a-r_i)\ge 8(a-r_i)\ge 6$ as
  $r_i,r_j\in [1-\frac{\sqrt{2}}{2},1+\frac{\sqrt{2}}{2}]\subseteq [0,2]$ and
  $a \in [3,5]$. Thus the aforementioned point also belongs to $G_j$.
  Similarly, it is easy to check that each coordinate is positive and
  upper bounded by $4 + 8^{n-1}\cdot2\cdot 5 \leq 8^{n+1}$.
  Consequently, the point belongs to $C_i^a$. Therefore, $\ell$ intersects $C_i^a$. 
\end{proof}

The following theorem immediately implies Theorem~\ref{t:compact_intro}.

\begin{theorem}
\label{t:compact}
  Let $\CC$ be the family of convex set in $\R^{3n}$ from
  Lemma~\ref{l:disjoint_sets_general}. Let $\LL$ be the space of
 lines transversal to $\CC$. Then $H_{n-1}(\LL)\ne 0$.
\end{theorem}
\begin{proof}
  Recall that $\CC = \{C_0\}\cup \{C_i^a\colon i\in [n],a\in \{3,4,5\}\}$ as in
  the statement of Lemma~\ref{l:disjoint_sets_general} and that $C_0
  \subseteq F_0$, where 
  \[
    F_0 = \left\{(x_1,x_1,z_1,x_2, x_2, z_2, \dots, x_n,x_n,z_n) \in \R^{3n}
    \colon \sum_{i=1}^n (z_i - 2x_i + 1)= \frac12\right\}
 \]
 as in~\eqref{e:F0}. In addition, it follows immediately from the definition of
  $C_i^a$ that $C_i^a\subseteq F_i^a$ whenever $i\in
  [n]$ and $a\in\{3,4,5\}$.

  Consequently, Lemma~\ref{l:general} implies that $\LL$ is a subset of

  \begin{multline}
 \label{e:L0}
  \LL_0 :=
  \bigg\{\big\{(r_1,r_1,r_1^2,\dots,r_n,r_n,r_n^2) +
  (0,q_1t,r_1q_1t,\dots,0,q_n t,r_n q_n t)\colon t \in
  \R\big\}\colon \\
    r_1, \dots, r_n \in \R \hbox{ are such that }
  \sum_{i=1}^{n}(r_i-1)^2=\frac{1}{2} \hbox{ and } q_1, \dots, q_n \in \R \setminus \{0\}\bigg\}.
  \end{multline}

Note that $\ell\in \LL_0$ meets $F_0$ exactly in one point, namely in the point
  $(r_1, r_1, r_1^2, \dots, r_n, r_n, r_n^2)$ for some $r_1, \dots, r_n \in \R$
  such that $\sum_{i=1}^{n}(r_i-1)^2=\frac{1}{2}$. Thus the map $f\colon \LL_0
  \to F_0$ given by $f(\ell):=\ell \cap F_0$ is continuous by
  Observation~\ref{o:project_to_F}. 

    In addition, $f(\LL_0)\subseteq R$ where
  \begin{equation}
   \label{e:R}
R:=\{(r_1,r_1,r_1^2,\dots,r_n,r_n,r_n^2)\colon\sum_{i=1}^n
  (r_i-1)^2=\frac{1}{2}\}\subseteq F_0.
  \end{equation}
    
    Next, let $S$ be the $(n-1)$-sphere 
\[
  S=\Big\{(r_1,\dots,r_n)\in \R^n\colon \sum_{i=1}^n (r_i-1)^2=\frac{1}{2}\Big\}.
\]
Let $g$ be the continuous map from $S$ to $\LL$ given by 
\[
g((r_1,\dots,r_n))=\{(r_1,r_1,r_1^2,r_2,r_2,r_2^2,\dots,r_n,r_n,r_n^2)+(0,8^{n-1}t,8^{n-1}r_1
  t,0,8^{n-2}t,8^{n-2}r_2 t,\dots,0,t,r_n t)\}.
\]
  Note that it follows from Lemma~\ref{l:disjoint_sets_general}, that the
  image of $g$ is indeed contained in $\LL$. Then $h:=f\circ g$ maps $(r_1,
  \dots, r_n)$ in $S$ to $(r_1, r_1, r_1^2, \dots, r_n, r_n, r_n^2) \in R$.
  Thus $h$ is a bijection between $S$ and $R$, and because $S$ is compact, $h$
  is a homeomorphism. From $f$ and $g$ we get induced maps on the homology
  $H_{n-1}(S)\to H_{n-1}(\LL) \to H_{n-1}(R)$. Then, $H_{n-1}(\LL)\ne
  0$ as $H_{n-1}(S)\cong H_{n-1}(R)\ne 0$ and the composition of the two
  induced maps is the identity.  
\end{proof}

\section{Inflating the sets---proof of Theorem~\ref{t:inflated_intro}}

Given a set $X \subseteq \R^d$ and $\delta > 0$, by $X^\delta$ we mean the open
$\delta$-neighborhood of $X$, that is, the set of points in distance less than
$\delta$ from a point of $X$.

\begin{lemma}
\label{l:close_transversal}
  Let $\KK = \{K_1, \dots, K_t\}$ be a family of pairwise disjoint compact convex sets in $\R^d$. 
Then for every $\varepsilon > 0$ there exists $\delta > 0$ such that for every
  transversal line $\ell_\delta$ to $\KK^\delta  := \{K_1^\delta, \dots,
  K_t^\delta\}$ there is a line $\ell$ transversal to $\KK$ such that for every
  $i \in [t]$, there are points $x \in \ell \cap K_i$ and $x_\delta \in
  \ell_\delta \cap K_i^\delta$ with $\dist(x, x_\delta) < \varepsilon$.
\end{lemma}

\begin{proof}
For contradiction assume that there exists $\varepsilon > 0$ such that for
  every $\delta > 0$ there is a line $\ell_\delta$ transversal to $\KK^\delta$
  and $i = i(\delta) \in [t]$ such that for every $x \in \ell \cap K_i$ and
  $x_\delta \in \ell_\delta \cap K_i^\delta$ we have $\dist(x, x_\delta) \geq
  \varepsilon$. We take a sequence $(\delta_k)_{k=1}^\infty$ with $\delta_k \to
  0$, and we pick $j \in [t]$ so that $j = i(\delta_k)$ infinitely often.
  Passing to a subsequence, with a slight abuse of the notation, we can assume
  that $j = i(\delta_k)$ for every positive integer $k$.

  Now, for every positive integer $k$, we take a point $y^k \in K_1 \times
  \cdots \times K_t$ so that its $i$th coordinate $y^k_i$ is some point in
  $\ell_{\delta_k} \cap K_i^{\delta_k}$. By compactness, there is an
  accumulation point $y = (y_1, \dots, y_t)$ for points $y^k$. Then $y_1,
  \dots, y_t$ lie on a line $\ell$. In addition, as $\bigcap_{k=1}^{\infty}
  K_i^{\delta_k} = K_i$, we get that $y_i \in K_i$ for $i \in [t]$ and
  consequently $\ell$ is transversal to $\KK$. On the
  other hand $\dist(y_j, y_j^k) \geq \varepsilon$ by our assumptions as $y_j
  \in \ell \cap K_j$ and $y_j^k \in \ell_{\delta_k} \cap K_j^{\delta_k}$. This
  contradicts that $y$ is an accumulation point of $y^k$.
\end{proof}

Throughout this section we assume that $n$ is a fixed integer, and we use the notation $F_i^a$, $F_0$, $C_i^a$, $C_0$
and $\CC$ as introduced in the previous section, including the statement of
Lemma~\ref{l:disjoint_sets_general}. We also reuse the notation for the set of
lines $\LL_0$ according to~\eqref{e:L0} and the set $R \subseteq F_0$ according
to~\eqref{e:R}.

Given $\varepsilon > 0$, by
$(C_i^a)^\varepsilon$ and $C_0^\varepsilon$ we denote the (convex, open) set of
all points in distance less than $\varepsilon$ from $C_i^a$ and $C_0$
respectively. We also set $\CC^\varepsilon := \{C_0^\varepsilon\} \cup
\{(C_i^a)^\varepsilon \colon i \in [n], a \in \{3,4,5\} \}$; that is,
$\CC^\varepsilon$ is obtained from $\CC$ by inflating each set there by
$\varepsilon$.

\begin{lemma}
\label{l:retract}
 There is $\eta > 0$ such that the set $R$ is a retract of $R^{\eta}$. That is, there is a continuous map
  $r\colon R^{\eta} \to R$ which is identity on $R$. 
\end{lemma}

\begin{proof}
  First we realize that $R$ is a smooth manifold (smoothly) embedded in
  $\R^{3n}$. Indeed consider the map $F \colon \R^{3n} \to \R^{2n+1}$ given by
\[
  F(x_1, y_1, z_1, \dots, x_n, y_n, z_n) := \Bigl(x_1 - y_1, x_1^2 - z_1, x_2 - y_2,
  x_2^2 - z_2, \dots, x_n - y_n, x_n^2 - z_n, \sum_{i=1}^n (x_i - 1)^2\Bigr).
\]
  Then, according to~\eqref{e:R}, $R$ is the level set $F^{-1}(0, \dots, 0,
  \frac12)$. It is easy to check that the matrix of partial derivatives of $F$
  has full rank for every point in $R$. Indeed, there are $2n$ independent
  vectors corresponding to $\frac{\partial}{\partial y_i}$ and
  $\frac{\partial }{\partial z_i}$, all of them zero in the last coordinate
  $F_{2n+1} = \sum_{i=1}^n (x_i - 1)^2$ while at least one of the
  vectors $\frac{\partial}{\partial x_1}, \dots, \frac{\partial}{\partial x_n}$
  is nonzero when applied to $F_{2n+1}$ for a point in $R$. Thus,
  by~\cite[Corollary~5.14]{lee13}, $R$ is indeed a smoothly embedded manifold in
  $\R^{3n}$.

  Now let $U$ be a tubular neighborhood of $R$ as in Chapter~6
  of~\cite{lee13}. According to~\cite[Proposition~6.25]{lee13}, $R$ is a retract of
  $U$. For sufficiently small $\eta > 0$, $R^\eta \subseteq U$, thus $R$ is a
  retract of $R^\eta$ as well.  
\end{proof}

Now we consider the hyperplane
\[
F'_0 := \{(x_1,y_1,z_1, \dots, x_x,y_n,z_n) \in \R^{3n} \colon x_1 = y_1\}.
\]
Note that $F_0 \subseteq F'_0$ according to~\eqref{e:F0}

In the proof of Theorem~\ref{t:compact}, we have checked that every line in
$\LL_0$ meets $F_0$ in exactly one point. It is not so difficult to check that
this is also true for $\LL_0$ and $F'_0$ (we will check it in the proof below).
It is crucial for us to transfer this
observation from $\LL_0$ to line transversals of $\CC^\varepsilon$ provided
that $\varepsilon > 0$ is small enough. 

\begin{lemma}
  \label{l:beta_delta}
  There are $\beta, \delta  > 0$ with the following properties. 

  \begin{enumerate}[(i)]
 \item
  For every line $\ell_\delta$ transversal to
  $\CC^\delta$ the angle between $\ell_\delta$ and $F'_0$ is at least $\beta$. In
  particular, $\ell_\delta$
  meets $F'_0$ in exactly one point.
\item Let $R$ be defined according to~\eqref{e:R}. Then every line
  $\ell_\delta$ transversal to $\CC^\delta$ meets $F'_0$ in $R^\eta$ where
      $\eta > 0$ is the value from Lemma~\ref{l:retract}.
\end{enumerate}
\end{lemma}

\begin{proof}
  We start with the proof of (i). First, we check that each line $\ell \in \LL_0$ meets $F'_0$ just in one
  point. Assume that $\ell$ has the equation as in~\eqref{e:L0}, considering
  $r_1, \dots, r_n, q_1, \dots, q_n$ fixed. An intersection point of $\ell$ and
  $F'_0$ is a point of form 
\[
  (r_1,r_1,r_1^2,\dots,r_n,r_n,r_n^2) +
  (0,q_1t,r_1q_1t,\dots,0,q_n t,r_n q_n t)
\]
 where $t \in \R$. Given that this intersection point belongs to $F'_0$, we get
  $t = 0$. Thus the intersection point is unique if it exists. On the other
  hand, if we set $t = 0$ we get a point that belongs both to $\ell$ and $F'_0$.

  By a compactness argument, similar to the proof of
  Lemma~\ref{l:close_transversal}, we realize that there is $\alpha > 0$ such
  that for every line $\ell$ transversal to $\CC$, the angle between $\ell$ and $F'_0$ is at
  least $\alpha$. Indeed, let $\alpha(\ell)$ be the angle between $\ell \in
  \LL_0$ and $F'_0$. For contradiction, assume that there is a sequence of
  lines $(\ell_k)_{k=1}^{\infty}$ transversal to $\CC$ such that
  $\alpha(\ell_k) \to 0$ as $k \to \infty$. (Note that a line $\ell$
  transversal to $\CC$ belongs to $\LL_0$ by Lemma~\ref{l:general}.) For every
  $C \in \CC$ and a positive integer $k$, consider a point $x^C_k \in \ell_k
  \cap C$. Recall that every $C \in \CC$ is compact. Thus the sequence
  $((x^C_k)\colon C \in \CC)_{k=1}^{\infty}$ has an accumulation point. By
  compactness, this accumulation point defines a line $\ell$ transversal to
  $\CC$ such that $\alpha(\ell) = 0$. This contradicts our earlier observation
  that $\ell$ intersects $F'_0$ in exactly one point.

  Now fix two sets $C_1, C_2 \in \CC$. Let $D > 0$ be the distance between
  $C_1$ and $C_2$. Take $\varepsilon > 0$ small enough where `small enough',
  depending only on $D$ and $\alpha$, 
  will be specified later on. Let $\delta$ correspond to $\varepsilon$
  according to Lemma~\ref{l:close_transversal} applied with the family $\CC$. 
  Let us consider a line $\ell_\delta$ transversal to
  $\CC^\delta$. Let $\ell$ be a line transversal to $\CC$ from the conclusion
  of Lemma~\ref{l:close_transversal}. For $i \in \{1,2\}$ let $x_i \in \ell \cap C_i$ and
  $x_{i,\delta} \in \ell_\delta \cap C^\delta_i$ be points such that
  $\dist(x_i,x_{i,\delta}) < \varepsilon$. Given that $\dist(x_1, x_2) \geq D$,
  $\varepsilon = \varepsilon(D, \alpha)$ can be set up so that the angle
  between $\ell$ and $\ell_\delta$ is less than $\alpha/2$. Thus the angle
  between $\ell_\delta$ and $F'_0$ is more than $\alpha - \alpha/2 = \alpha/2$.
  Then it is sufficient to set $\beta := \alpha/2$. This finishes the proof of
  (i).

  For (ii) we first set $\varepsilon' > 0$ so that $\varepsilon'(1 + 1/\tan
  (\beta)) < \eta$. Now we take the corresponding $\delta' > 0$ from
  Lemma~\ref{l:close_transversal} with respect to the family $\CC$. We can
  assume that $\delta < \delta'$ (where $\delta$ is from item (i)) as if the
  conclusion of item (i) is true for some value of $\delta$, then it
  is also true for any smaller value (thus we can decrease $\delta$ in (i) if
  necessary). Let $\ell_\delta$ be a transversal to
  $C^\delta$ (in particular a transversal to $C^{\delta'}$) and let $\ell$ be the
  corresponding transversal to $\CC$ from Lemma~\ref{l:close_transversal}
  applied with $\delta'$ and $\varepsilon'$ as above. Let $x_0 \in \ell \cap
  C_0$ and  $x_{0,\delta'} \in \ell_\delta \cap C_0^{\delta'}$ be points such that
  $\dist(x_0, x_{0,\delta'}) < \varepsilon'$. Let $y$ be the orthogonal
  projection of $x_{0,\delta'}$ to $F'_0$ and let $z$ be the intersection point
  of $\ell_\delta$ and $F'_0$; see Figure~\ref{f:distance}. 

\begin{figure}
    \centering       
    \includegraphics{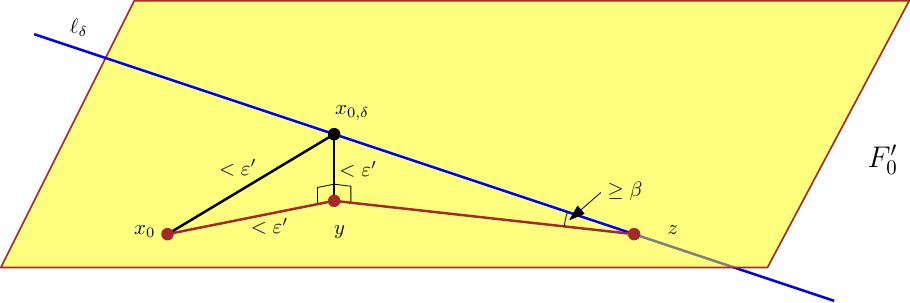}
    \caption{The points $x_0, x_{0,\delta}, y$ and $z$. We depict only the
    affine subspace spanned by these points.}
    \label{f:distance}      
\end{figure}

  Now we have that $x_0 \in R$ as in the proof of Theorem~\ref{t:compact} because
  $x_0$ is the intersection point of $\ell$ and $F_0$. We also have $\dist(x_0,
  y) < \varepsilon'$ as $\dist(x_0, x_{0,\delta'}) < \varepsilon'$ and $x_0 \in
  F'_0$. Finally, we have $\dist(y,z) \leq \varepsilon' / \tan(\beta)$ as
  $x_{0,\delta'}$ and $z$ belong to $\ell_\delta$, $\dist(x_{0,\delta'},y) \leq
  \varepsilon'$ and the angle between $\ell_\delta$ and $F'_0$ is at least
  $\beta$. Thus $\dist(x_0, z) \leq \varepsilon'(1 + 1/\tan(\beta)) < \eta$.
  Thus $z \in R^{\eta}$ as required. 
\end{proof}

The following theorem immediately implies Theorem~\ref{t:inflated_intro}.

\begin{theorem}
\label{t:delta_neighborhood}
  There is $\delta > 0$ such that $H_{n-1}(\LL)\ne 0$ where $\LL$ is the space of
 lines transversal to $\CC^\delta$.
\end{theorem}

\begin{proof}
  Our desired value $\delta$ is the value from Lemma~\ref{l:beta_delta}. By
  this lemma, every $\ell$ transversal to $\CC^\delta$ meets $F'_0$ exactly in
  one point.
  Thus the map $f\colon \LL
  \to F'_0$ given by $f(\ell):=\ell \cap F'_0$ is continuous by
  Observation~\ref{o:project_to_F}. 

    In addition, $f(\LL)\subseteq R^{\eta}$ by the item (ii) of
    Lemma~\ref{l:beta_delta} where $\eta$ is the value from
    Lemma~\ref{l:retract}.
    Next, let $S$ be the $(n-1)$-sphere 
\[
  S=\Big\{(r_1,\dots,r_n)\in \R^n\colon \sum_{i=1}^n (r_i-1)^2=\frac{1}{2}\Big\}.
\]
and let $g$ be the continuous map from $S$ to $\LL$ given by 
\[
g((r_1,\dots,r_n))=\{(r_1,r_1,r_1^2,r_2,r_2,r_2^2,\dots,r_n,r_n,r_n^2)+(0,8^{n-1}t,8^{n-1}r_1
  t,0,8^{n-2}t,8^{n-2}r_2 t,\dots,0,t,r_n t)\}
\]
  defined the same way as in the proof of Theorem~\ref{t:compact}.
  It follows from Lemma~\ref{l:disjoint_sets_general}, that the
  image of $g$ is contained in the set of transversals to $\CC$, a fortiori, it
  is contained in $\LL$. Let $r \colon R^\eta \to R$ be the retraction map from
  Lemma~\ref{l:retract}.

  Then $f\circ g$ maps $(r_1,
  \dots, r_n)$ in $S$ to $(r_1, r_1, r_1^2, \dots, r_n, r_n, r_n^2) \in R
  \subseteq R^\eta$. Consequently,  $h := r \circ f \circ g$ is a bijection between $S$ and $R$, and because $S$ is compact, $h$
  is a homeomorphism. From $r$, $f$ and $g$ we get induced maps on the homology
  $H_{n-1}(S)\to H_{n-1}(\LL) \to H_{n-1}(R^\eta) \to H_{n-1}(R)$. Then, $H_{n-1}(\LL)\ne
  0$ as $H_{n-1}(S)\cong H_{n-1}(R)\ne 0$ and the composition of the three
  induced maps is the identity.  
\end{proof}

\bibliographystyle{alpha} 
\bibliography{references}

\newcommand{\etalchar}[1]{$^{#1}$}
\begin{thebibliography}{CdVGG14}

\bibitem[ABG26]{abg26}
Sergey Avvakumov, Marguerite Bin, and Xavier Goaoc.
\newblock Intersection patterns of set systems on manifolds with slowly growing
  homological shatter functions.
\newblock {\em arXiv preprint arXiv:2601.02920}, 2026.

\bibitem[BGP08]{bgp08}
Ciprian Borcea, Xavier Goaoc, and Sylvain Petitjean.
\newblock Line transversals to disjoint balls.
\newblock {\em Discrete Comput. Geom.}, 39(1-3):158--173, 2008.

\bibitem[CdVGG14]{cgg14}
\'Eric Colin~de Verdi\`ere, Gr\'egory Ginot, and Xavier Goaoc.
\newblock Helly numbers of acyclic families.
\newblock {\em Adv. Math.}, 253:163--193, 2014.

\bibitem[CGH24]{cgh24}
Otfried Cheong, Xavier Goaoc, and Andreas~F. Holmsen.
\newblock Some new results on geometric transversals.
\newblock {\em Discrete Comput. Geom.}, 72(2):674--703, 2024.

\bibitem[GP88]{gp88}
Jacob~E. Goodman and Richard Pollack.
\newblock Hadwiger's transversal theorem in higher dimensions.
\newblock {\em J. Amer. Math. Soc.}, 1(2):301--309, 1988.

\bibitem[GPP{\etalchar{+}}17]{gpptw17}
Xavier Goaoc, Pavel Pat\'ak, Zuzana Pat\'akov\'a, Martin Tancer, and Uli
  Wagner.
\newblock Bounding {H}elly numbers via {B}etti numbers.
\newblock In {\em A journey through discrete mathematics}, pages 407--447.
  Springer, Cham, 2017.

\bibitem[Had57]{hadwiger57}
H.~Hadwiger.
\newblock Ueber {E}ibereiche mit gemeinsamer {T}reffgeraden.
\newblock {\em Portugal. Math.}, 16:23--29, 1957.

\bibitem[Lee13]{lee13}
John~M. Lee.
\newblock {\em Introduction to smooth manifolds}, volume 218 of {\em Graduate
  Texts in Mathematics}.
\newblock Springer, New York, second edition, 2013.

\bibitem[Mon14]{montejano14}
Luis Montejano.
\newblock A new topological {H}elly theorem and some transversal results.
\newblock {\em Discrete Comput. Geom.}, 52(2):390--398, 2014.

\bibitem[MS26]{ms26}
Daniel McGinnis and Nikola Sadovek.
\newblock A necessary and sufficient condition for {$k$}-transversals.
\newblock {\em Adv. Math.}, 490:Paper No. 110829, 13, 2026.

\bibitem[PW90]{pw90}
R.~Pollack and R.~Wenger.
\newblock Necessary and sufficient conditions for hyperplane transversals.
\newblock {\em Combinatorica}, 10(3):307--311, 1990.

\bibitem[Wen90]{wenger90}
Rephael Wenger.
\newblock A generalization of {H}adwiger's transversal theorem to intersecting
  sets.
\newblock {\em Discrete Comput. Geom.}, 5(4):383--388, 1990.

\end{thebibliography}

\end{document}